\documentclass[11pt]{article}
\usepackage{amsmath,amssymb,amsthm, epsfig}
\usepackage{amsfonts}
\usepackage{amsmath}
\usepackage{amssymb}
\usepackage{color}
\usepackage[mathscr]{eucal}

\title{Optimal approximate fixed point results in locally convex
spaces}

\author{\textsc{C. S. Barroso}\footnote{Supported in part by CNPq-Brazil, Grant 307210/2009-0.},\quad
\textsc{O. F. K. Kalenda}\footnote{Supported by the Research grant GA \v{C}R P201/12/0290.}\quad and\quad \textsc{M. P. Rebou\c cas}\footnotemark[1]}

\newlength{\hchng}
\newlength{\vchng}
\def \N {\mathbb{N}}
\def \R {\mathbb{R}}

\def \dist {\mathrm{dist}}

\def \T {\mathscr{T}}
\def\spann{\hbox{\rm span}\,}

\def\co{{\rm{co}}}

\newtheorem{theorem}{Theorem}[section]
\newtheorem{lemma}[theorem]{Lemma}
\newtheorem{proposition}[theorem]{Proposition}
\newtheorem{corollary}[theorem]{Corollary}
\newtheorem{question}[theorem]{Problem}
\theoremstyle{definition}
\newtheorem{definition}[theorem]{Definition}
\newtheorem{example}[theorem]{Example}

\numberwithin{equation}{section}

\newcommand{\intav}[1]{\mathchoice {\mathop{\vrule width 6pt height 3 pt depth  -2.5pt
\kern -8pt \intop}\nolimits_{\kern -6pt#1}} {\mathop{\vrule width
5pt height 3  pt depth -2.6pt \kern -6pt \intop}\nolimits_{#1}}
{\mathop{\vrule width 5pt height 3 pt depth -2.6pt \kern -6pt
\intop}\nolimits_{#1}} {\mathop{\vrule width 5pt height 3 pt depth
-2.6pt \kern -6pt \intop}\nolimits_{#1}}}

\date{}

\begin{document}
\maketitle

\begin{abstract}
Let $C$ be a convex subset of a locally convex space. We provide  optimal approximate fixed point results for sequentially continuous maps $f\colon C\to\overline{C}$. First we prove that if $f(C)$ is totally bounded, then it has an approximate fixed point net. Next, it is shown that if $C$ is bounded but not totally bounded, then there is a uniformly continuous map $f\colon C\to C$ without approximate fixed point nets. We also exhibit an example of a sequentially continuous map defined on a compact convex set with no approximate fixed point sequence. In contrast, it is observed that every affine (not-necessarily continuous) self-mapping a bounded convex subset of a topological vector space has an approximate fixed point sequence. Moreover, it is constructed a affine sequentially continuous map from a compact convex set into itself without fixed points.

\medskip

\noindent \textit{MSC:} 47H10, 54H25.

\medskip

\noindent \textbf{Keywords:} Approximate fixed point net, approximate fixed point sequence, totally bounded set, sequentially continuous map, uniformly continuous map.

\end{abstract}


\section{Introduction}\label{sec:1}

The well-known Schauder-Tychonoff theorem says that any continuous selfmap of a nonempty compact convex subset of a Hausdorff locally convex space has a fixed point. Since fixed point theorems have many
applications in the theory of differential equations and elsewhere, the search for new ones continues. There are several types of such results -- one can relax the compactness assumption, the continuity assumption or the conclusion that there is a fixed point.

If one relaxes the compactness assumption, the assumption of continuity must be {strengthened}. Indeed, already Klee \cite{klee} proved that any noncompact convex subset of a normed space admits a fixed point free continuous selfmap. On the other hand, there are many noncompact convex subsets of Banach spaces which have the fixed point property for $1$-Lipschitz (i.e., nonexpansive) maps. (There is a large theory concerning this fact.) For Lipschitz maps the situation is different -- by a {strengthening} of the above mentioned result of Klee due to Lin and Sternfeld \cite{LS} no noncompact convex subset of a normed space has the fixed point property for Lipschitz maps.

Another possibility is to relax the continuity assumption. We will pay attention to the sequential continuity. In \cite[Theorem 1]{AGP} the authors show that if $C$ is a nonempty weakly compact subset of a metrizable locally convex space $X$, then any weakly sequentially continuous selfmap of $C$ has a fixed point. This theorem has a lot of {applications} in the theory of differential equations and the authors ask whether the metrizability assumption can be dropped. The proof of their theorem is based on showing that any weakly sequentially continuous selfmap of $C$ is in fact weakly continuous (due to the angelicity of the weak topology of metrizable locally convex spaces) and reducing the result to the Schauder-Tychonoff theorem.

We show, in particular, that {sequential} continuity is very weak condition. On the positive part, we show that the metrizability condition can be dropped if $C$ is sequentially compact and $f$ is affine. On the other hand, if $C$ is compact and $f$ sequentially continuous and affine, $f$ need not have any fixed point. The underlying space is $(\ell_\infty^*,w^*)$ and an essential role is played, of course, by the Grothendieck property of $\ell_\infty$.

So, sequential continuity is too weak to ensure the existence of fixed points. However, we show that for approximate fixed points it is strong enough. Let us recall the definitions.

 Let $C$ be a nonempty convex subset of a topological vector space $X$. An approximate fixed point sequence for a map $f\colon C\to\overline{C}$ is a sequence $(x_n)$ in $C$ so that $x_n- f(x_n)\to 0$. Similarly, we can define approximate fixed point nets for $f$. Let us mention that $f$ has an approximate fixed point net if and only if
$$0\in\overline{\{x-f(x):x\in C\}}.$$
Finding approximate fixed points for continuous maps is an interesting topic in fixed point theory, and a great number of results have been proposed in the literature, among which we mention e.g. Hazewinkel and van de Vel \cite{HV}, Had\v{z}i\'c \cite{Hadzic1,Hadzic2}, Idzik \cite{Idzik1,Idzik2}, Lin-Sternfeld \cite{LS} and Park \cite{Park2}. We refer also the reader to \cite{Barroso, Barroso-Lin,BKL} for several results related to the weak-approximate fixed point sequences for continuous maps.

In particular, in \cite{LS} it is proved that a convex subset $C$ of a normed space is totally bounded if and only if any Lipschitz selfmap of $C$ has an approximate fixed point sequence.
A generalization of the `only if' part is a result of \cite{Park2} saying, in particular, that any continuous selfmap of a convex subset of a locally convex space with totally bounded range has an approximate fixed point net.

In the present paper, we first extend the result of \cite{Park2} to sequential continuous mappings. This extension is allowed by the fact that the problem can be reduced to a finite-dimensional problem and in a finite-dimensional space sequential continuity implies continuity. We further show that sequential continuity does not help in obtaining approximate fixed point sequences even for selfmaps of compact convex sets.

Finally, we extend the 'if part' of the above quoted result of \cite{LS} to locally convex spaces. We show that any nonempty bounded convex subset of a locally convex space which is not totally bounded admits a uniformly continuous selfmap without approximate fixed point net.

\section{The results}\label{sec:2}

Our main positive result is the following theorem.

\begin{theorem}\label{thm:M1} Let $C$ be a convex subset of a locally convex space $X$ and $f\colon C\to \overline{C}$ be a sequentially continuous map such that $f(C)$ is totally bounded.
Then $f$ has an approximate fixed point net.
\end{theorem}

The same result, under the stronger assumption that $f$ is continuous, was proved by Park in \cite[Corollary 2.1]{Park2}.
We extend his result to the case of sequentially continuous $f$.
The theorem follows immediately from Theorem~\ref{thm:M3} below
which is a common generalization of \cite[Theorem 2.1]{BKL} and \cite[Corollary 2.1]{Park2}.

In case $X$ is metrizable (or, more generally, Fr\'echet-Urysohn), we get even an approximate fixed point sequence. In general, an approximate fixed point sequence need not exist -- even if $C$ is compact (see Example~\ref{ex:1} below) or if $f$ is continuous (this follows from \cite[Example 4.2]{BKL} if we observe that in the weak topology any bounded set is totally bounded).

As a consequence we get the following strengthening of a result of \cite{LS}:

\begin{corollary}\label{cor:LS} Let $C$ be a totally bounded convex subset of a normed space $X$. Then every continuous map $f:C\to C$ admits an approximate fixed point sequence.
\end{corollary}

This is already a consequence of \cite[Corollary 2.1]{Park2}. However, we find interesting to formulate it explicitly, as it was not done in \cite{Park2}.

Let us comment this corollary a bit. The main result of \cite{LS} is their Theorem 1 which says, in particular, that any Lipschitz selfmap $f$ of a convex subset $C$ of a normed space $X$ has an approximate fixed point sequence if and only if $C$ is totally bounded. Here, the `if' part is the easy implication. Indeed,
$f$ can be extended to the closure of $C$ in the completion of $X$.
This closure is compact, hence the extended function has {a} fixed point by Schauder's theorem, so we get an approximate fixed point sequence for $f$ (see \cite[p. 635]{LS}). The above corollary extends the easy implication of \cite{LS} to continuous (not necessarily Lipschitz) functions.

Theorem~\ref{thm:M1} is the best possible result in a sense.
It is witnessed by the following two results. The first one says
that the assumption of total boundedness is essential, even for uniformly continuous maps. The second one shows that sequentially continuous maps need not have approximate fixed point sequences even if $C$ is compact.

\begin{theorem}\label{thm:M2}
Let $X$ be a locally convex space and $C\subset X$ a bounded convex set which is not totally bounded. Then there is a uniformly continuous map $f:C\to C$ with no approximate fixed point net.
\end{theorem}

This result will be proved in the last section. It is a kind of extension of the result of \cite{LS}. Let us recall that in \cite{LS} it is proved that any convex subset of a normed space which is not totally bounded admits a Lipschitz selfmap without approximate fixed point net. Since our result concerns general locally convex spaces, we replace Lipschitz maps by uniformly continuous ones. However, the resulting selfmap is also `Lispchitz' in a sense, see the formulation of Theorem~\ref{thm:M2-1} below.

We continue by the promised example of a {sequentially} continuous selfmap of a compact convex set without approximate fixed point sequence.

\begin{example}\label{ex:1} There is a Hausdorff locally convex space $X$ equipped with its weak topology, a nonempty compact convex subset $K\subset X$
and a sequentially continuous function $f:K\to K$ with no approximate fixed point sequence.
\end{example}

\begin{proof} Let $X=(\ell_\infty^*,w^*)$ and
$K=\{\mu\in X:\mu\ge 0\ \&\ \|\mu\|\le 1\}$.
Then $X$ is a locally convex space, the topology is its weak one
and $K$ is a nonempty convex compact subset of $X$. It remains
to construct the function $f$.

The space $\ell_\infty^*$ can be canonically identified with the space $M(\beta\N)$ of signed Radon measures on the compact space $\beta\N$ (\v{C}ech-Stone {compactification} of natural numbers).
Let $P:M(\beta\N)\to M(\beta\N)$ be defined by
$$P(\mu)(A)=\mu(A|_\N)=\sum_{n=1}^\infty \mu(\{n\})\delta_n,$$
{where $\delta_x$ denotes the Dirac measure supported by $x$.} Then $P$ is a bounded linear operator. We set
$K_0=P(K)$. Then $K_0\subset K$ and $K_0$ is a convex subset of $\ell_\infty^*$ which is not totally bounded in the norm. Hence, by
\cite[Theorem 1]{LS} there is a Lipschitz map $g:K_0\to K_0$ without
approximate fixed point sequence (with respect to the norm).

Set $f= g\circ P|_K$. We claim that $f$ is weak$^*$-to-weak$^*$ sequentially continuous and has no approximate fixed point sequence
in the weak* topology.

To show the first assertion, let $(\mu_n)$ be a sequence in $K$ weak$^*$ converging to some $\mu\in K$. Since $\ell_\infty$ is a Grothendieck space, $\mu_n\to\mu$ weakly in $\ell_\infty^*$. Since $P$ is a bounded linear operator, it is also weak-to-weak continuous, hence $P\mu_n\to P\mu$ weakly in $\ell_\infty^*$.
Since $P(\ell_\infty^*)$ is isometric to the space $\ell_1$, by the Schur property we have $P\mu_n\to P\mu$ in the norm, so
$g(P\mu_n)\to g(P\mu)$ in the norm. We conclude that $f(\mu_n)\to f(\mu)$ in the norm, hence also in the weak$^*$ topology.
This completes the proof that $f$ is sequentially continuous.

Next suppose that $(\mu_n)$ is an {approximate} fixed point sequence in $K$. Then $\mu_n-f(\mu_n)\to 0$ in the weak$^*$ topology. By the Grothendieck property of $\ell_\infty$ we get that $\mu_n-f(\mu_n)\to 0$ weakly in $\ell_\infty^*$. Since $P$ is a bounded linear operator, we get $P\mu_n-P f(\mu_n)\to 0$ weakly,
so $P\mu_n-Pf(\mu_n)\to 0$ in the norm by the Schur theorem.
Further,
$$P\mu_n-Pf(\mu_n)= P\mu_n-f(\mu_n)=P\mu_n- g(P\mu_n),$$
so $(P\mu_n)$ is an approximate fixed point sequence for $g$ with respect to the norm. It is a contradiction completing the proof.
\end{proof}

We complete the picture by the following results on affine maps.
We include the following easy observations on approximate fixed point sequences and fixed points of affine maps.

\begin{proposition}\label{prop:1} Let $X$ be a topological vector space, $C\subset X$ a nonempty bounded convex set and $f:C\to C$ an affine selfmap. Then the following assertions hold:
\begin{itemize}
	\item[(i)] The mapping $f$ has an approximate fixed point sequence.
	\item[(ii)] If $X$ is Hausdorff, $C$ is countably compact and $f$ is continuous, then $f$ has a fixed point.
	\item[(iii)] If $X$ is Hausdorff, $C$ is sequentially compact and $f$ is sequentially continuous, then $f$ has a fixed point.
\end{itemize}
\end{proposition}

\begin{proof} (i) Fix any $y_1\in C$ and define inductively the sequence $(y_k)$ by setting $y_{k+1}=f(y_k)$. Set
$$x_k=\frac{y_1+\dots+y_k}{k}.$$
Then
$$x_k-f(x_k)=\frac{y_1-y_{k+1}}{k}\to0$$
as $C$ is bounded.

(ii) Let $(x_k)$ be an approximate fixed point sequence given by the assertion (i). Since $C$ is countably compact, there is some $x\in C$ which is a cluster point of $(x_k)$, hence there is some subnet $(x_\nu)$ of $(x_k)$ which converges to $x$. By continuity of $f$ we get $f(x_\nu)\to f(x)$. However, $(x_\nu-f(x_\nu))$ is a subnet of $(x_k-f(x_k))$, hence $x_\nu-f(x_\nu)\to 0$. So, $x=f(x)$.

(iii) The proof can be done in the same way as that of (ii), we only use sequential compactness to extract a subsequence $(x_{k_n})$ converging to some $x\in C$ and then we use sequential continuity to deduce that $f(x_{k_n})\to f(x)$.
\end{proof}

Let us stress that in the assertion (i) no continuity property of $f$ is assumed. This observation generalizes \cite[Theorem 4]{GJZ},
where the assertion (i) is proved for nets.

The assertion (ii) under the stronger assumption that $C$ is compact is well known. It is a special case of the Schauder-Tychonoff theorem,
but the proof is easy, it does not require the use of Brouwer's theorem.

The assertion (iii) says that in some very special cases sequentially continuous maps have fixed points.  However,
the assumption of sequential compactness cannot be replaced by compactness, as witnessed by the following example shows that sequentially continuous affine selfmaps of compact convex sets need not have a fixed point.

\begin{example}\label{ex:2} There is a Hausdorff locally convex space $X$ equipped with its weak topology, a nonempty compact convex subset $K\subset X$
and an affine sequentially continuous function $f:K\to K$ with no  fixed point.
\end{example}

\begin{proof} Let $X=(\ell_\infty^\star, w^\star)$. We can regard $X$ as signed Radon measures on $\beta\mathbb{N}$. Let $C$ be the subset of $X$ consisting of probability measures. Then $C$ is compact and convex. Now, pick a decomposition $\{A_n\colon n\in\mathbb{N}\}$ of $\mathbb{N}$ into infinite disjoint subsets. Next, we shall use this decomposition to define a sequence $(k_j)$ of natural numbers as follows:
\begin{itemize}
\item[(a)] $k_1\geq 2$ and $k_1\not\in A_1$.
\item[(b)] $k_{j+1}>k_j$, and $k_{j+1}\not\in A_1\cup A_2\cup\dots\cup A_{j+1}$ for each $j\in\N$.
\end{itemize}
Let us define now a linear map $f\colon X\to X$ by the formula
\[
f(\mu)=\mu(\beta\mathbb{N}\setminus \mathbb{N})\cdot\delta_1+\sum_{j=1}^\infty \mu(A_j)\cdot\delta_{k_j},
\]
where $\delta_x$ denotes the Dirac measure supported by $x$. Then $f$ is a linear mapping which is norm-to-norm continuous, and hence weak-to-weak continuous on $\ell_\infty^\star$. As $\ell_\infty$ is a Grothen\-dieck space, $f$ is weak$^\star$-to-weak$^\star$ sequentially continuous.
In other words, it is sequentially continuous when considered from $X$ to $X$. Further,  it is obvious that $f$ fixes $C$.

Finally, $f$ has no fixed point in $C$. Indeed, suppose that $\mu\in C$ is a fixed point, i.e., $f(\mu)=\mu$. Since $f(\mu)$ is supported by $\N$, we have
$$\mu(\{1\})=f(\mu)(\{1\})=\mu(\beta\N\setminus\N)=f(\mu)(\beta\N\setminus \N)=0.$$
Hence $\mu$ is supported by the set $\{k_j:j\in\N\}$. Since $\mu$ is a probability measure, we can find the minimal $j$ such that
$\mu(\{k_j\})\ne 0$. However,
$$\mu(\{k_j\})=f(\mu)(\{k_j\})=\mu(A_j)=0$$
as $k_l\notin A_j$ for $l\ge j$ by the condition (b).
It is a contradiction.
\end{proof}

Let us remark that it seems not to be clear whether the assumption that $f$ is affine is essential in the assertions (ii) and (iii) of Proposition~\ref{prop:1}. The proof given above works provided $f$ admits an approximate fixed point sequence. However, the best thing we are able to obtain is an approximate fixed point by Theorem~\ref{thm:M1}. Indeed, countably compact sets in topological vector spaces are necessarily totally bounded. We have not found
a reference for this easy observation, so we include a proof.

\begin{proposition}\label{prop:cc} Let $X$ be a topological vector space and $C\subset X$ a relatively countably compact subset. Then $C$ is totally bounded.
\end{proposition}

\begin{proof} The proof will be done by contraposition. Suppose that $C$ is not totally bounded. It means that there is $U$, a balanced neighborhood of zero, such that $C$ cannot be covered by finitely many translates of $U$. We can then construct by induction a sequence $(x_n)$ in $C$ such that for each $n\in\N$ we have
$$x_{n+1}\notin\{x_1,\dots,x_n\}+U.$$
Then the set $A=\{x_n:n\in\N\}$ is a closed discrete subset of $X$. Indeed, let $V$ be a balanced neighborhood of $0$ such that $V+V\subset U$. Then for any $x\in X$ the set $x+V$ contains at most one element of $A$. Indeed, suppose that $m<n$ and $\{x_m,x_n\}\subset x+V$. Then $x\in x_m+V$, thus
$$x_n\in x+V\subset x_m+V+V\subset x_m+U,$$
a contradiction.
It follows that $A$ is an infinite subset of $C$ without accumulation point in $X$. Therefore $C$ is not relatively countably compact.
\end{proof}

We finish this section by the following questions which are, up to our knowledge, open.

\begin{question} Let $X$ be a Hausdorff locally convex space, $C\subset X$ a convex set and $f:C\to C$ a mapping. Suppose that one of the following two conditions is satisfied:
\begin{itemize}
	\item $C$ is countably compact and $f$ is continuous.
	\item $C$ is sequentially compact and $f$ is sequentially continuous.
\end{itemize}
Does $f$ necessarily admit a fixed point?
\end{question}

It follows from Proposition~\ref{prop:cc} and Theorem~\ref{thm:M1} that $f$ has an approximate fixed point net
provided $C$ is countably compact and $f$ is sequentially continuous. However, in this case $f$ need not have an approximate fixed point sequence (even if $C$ is compact, see Example~\ref{ex:1}). And, even if $f$ admits an approximate fixed point sequence, it need not have a fixed point (see Example~\ref{ex:2}).
It follows that the above question is natural, as in the quoted examples the respective sets are not sequentially compact and the respective maps are not continuous.

Let us remark that we would get the positive answer to the above questions if we are able to construct an approximate
fixed point sequence (see the proof of Proposition~\ref{prop:1}). However, we have no idea how to do that.

Finally, let us remark that there are some special cases when the answers are positive. For example, if $X$ is angelic,
then countably compact sets are compact and sequentially continuous on them maps are continuous. Another possibility is to strengthen the first version of assumptions -- to assume moreover that countable subsets of $C$ have compact closures in $C$.

\section{Proof of Theorem \ref{thm:M1}}\label{sec:3}

In this section we will prove a more general Theorem~\ref{thm:M3}. Theorem~\ref{thm:M1} then immediately follows. Theorem~\ref{thm:M3} is in fact a common generalization of Theorem~\ref{thm:M1} and \cite[Theorem 2.1]{BKL}. To formulate the theorem we need the following definition due to Himmelberg \cite{Himmelberg}:

\begin{definition} A nonempty subset $C$ of a topological vector space $X$ is said to be almost convex if for any neighborhood $V$ of the origin $0$ in $X$ and for any finite set $\{x_1,\dots, x_n\}\subset C$, there exists a finite set $\{z_1,\dots,z_n\}\subset C$ so that
\[
\co\{z_1,\dots,z_n\}\subset C\quad  \text{and}\quad z_i - x_i\in V\quad \text{for all}\quad i=1,\dots, n.
\]
\end{definition}

\begin{theorem}\label{thm:M3} Let $C$ be an almost convex subset of a topological vector space $(X,\mathscr{T})$, let $\mathscr{R}$ be a weaker locally convex topology on $X$ and $f\colon C\to \overline{C}$ be a $\mathscr{T}$-to-$\mathscr{R}$ sequentially continuous map such that
$f(C)$ is $\mathscr{R}$-totally bounded. Then $f$ has an approximate fixed point net.
\end{theorem}

This theorem is an obvious consequence of the following lemma.

\begin{lemma}\label{L-seminorm} Let $C$ be an almost convex subset of a topological vector space $(X,\mathscr{T})$, let $\rho$ be a continuous seminorm on $X$ and $f\colon C\to \overline{C}$ be a $\mathscr{T}$-to-$\rho$ sequentially continuous map such that
$f(C)$ is $\rho$-totally bounded. Then for each $\varepsilon>0$ there is $x\in C$ with $\rho(x-f(x))<\varepsilon$.
\end{lemma}

The rest of this section is devoted to the proof of Lemma~\ref{L-seminorm}. We will use the following lemma, which is a slight generalization of a result of K.~Fan (see \cite[proof of Lemma 1]{Fan61} or \cite[Lemma]{Park}).

\begin{lemma}\label{lem:1} Let $C$ be a subset of a topological {vector} space $(X,\mathscr{T})$, $D$ a nonempty finite subset of $C$ such that $\co D\subset C$, and $F\colon D\to 2^C$ a multivalued map with the following two properties:
\begin{itemize}
\item[(a)] $F(z)$ is sequentially closed for all $z\in D$.
\item[(b)] $\co N \subset \bigcup_{z\in N} F(z)$ for all $N\subset D$.
\end{itemize}
 Then $\bigcap_{z\in D} F(z)\neq \emptyset$.
\end{lemma}

The only generalization consists in assuming that the values $F(z)$ are sequentially closed (not necessarily closed). But Fan's proof obviously works for sequentially closed values as well (as a continuous preimage of a sequentially closed set is sequentially closed and sequentially closed sets in $\R^n$ are closed). However, this easy modification allows us to proof Lemma~\ref{L-seminorm}:

\begin{proof}[Proof of Lemma~\ref{L-seminorm}.]
Let $\varepsilon>0$ be arbitrary. Since $f(C)$ is $\rho$-totally bounded,
and $f(C)\subset\overline C$, there is a finite set $\{x_1,\dots, x_n\}\subset C$ such that for any $x\in f(C)$ there is some $i\in\{1,\dots,n\}$ with $\rho(x-x_i)<\frac\varepsilon2$.
Since $C$ is almost convex, we can also find points $z_1,\dots, z_n$ in $C$ so that $\rho(z_i-x_i)<\frac{\varepsilon}{2}$ for each $i=1,2,\dots, n$, and $\co\{z_1,\dots, z_n\}\subset C$.
Now set $D=\{z_1,\dots,z_n\}$ and define a multimap $F\colon D\to 2^C$ by putting for each $i$,
\[
F(z_i)=\left\{ x\in C\colon \rho(f(x)-x_i)\ge \frac{\varepsilon}{2}\right\}.
\]
Since $\rho$ is continuous and $f$ is sequentially continuous, each $F(z_i)$ is sequentially closed in $C$. Moreover, we have
\[
\bigcap_{i=1}^n F(z_i)=\emptyset.
\]
This follows from the choice of $x_1,\dots,x_n$.  By Lemma \ref{lem:1} applied to $F$, $D$ and $C$, we conclude that there exist a subset $\{z_{\kappa_1},z_{\kappa_2},\dots, z_{\kappa_m}\}$ of $D$ and an $x\in \co \{z_{\kappa_1},z_{\kappa_2},\dots, z_{\kappa_m}\}$ such that $x\notin \bigcup_{j=1}^m F(z_{\kappa_j})$. Hence $\rho(f(x)-x_{\kappa_j})<\frac{\varepsilon}{2}$ for all $j=1,2,\dots, m$,
so by the triangle inequality
\begin{equation}\label{eqn:2}
\rho(f(x)-z_{\kappa_j})<\varepsilon\quad \text{for all } j=1,2,\dots, m.
\end{equation}
Since $x\in\co \{z_{\kappa_1},z_{\kappa_2},\dots, z_{\kappa_m}\}$,
we get by using again the triangle inequality
$$\rho(f(x)-x)<\varepsilon.$$
This completes the proof.
\end{proof}

\medskip
\noindent{\bf Remark.} The proof of Lemma~\ref{L-seminorm} is inspired by \cite{Park}. Another possibility would be to follow the lines of the proof
of \cite[Theorem 2.1]{BKL} and to reduce the problem to Brouwer's fixed point theorem. Lemma~\ref{lem:1} is a substitute for Brouwer's theorem -- it is a consequence of the famous Knaster-Kuratowski-Mazurkiewicz principle (see \cite{KKM}).

\section{Proof of Theorem~\ref{thm:M2}}\label{sec:4}

Theorem~\ref{thm:M2} will be proved using results and ideas of \cite{LS}. We will prove a slightly stronger result:

\begin{theorem}\label{thm:M2-1}
Let $(X,\T)$ be a locally convex space and $C\subset X$ a bounded convex set. Let $\rho_0$ be a continuous seminorm such that $C$ is not $\rho_0$-totally bounded. Then there is a mapping $f:C\to C$ with the following properties:
\begin{itemize}
	\item $f$ admits no approximate fixed point net.
	\item For any continuous seminorm $\rho$ the mapping $f$ is $\rho_0$-to-$\rho$ Lipschitz.
\end{itemize}
\end{theorem}

The key role in the proof of this theorem is played by the space $\Delta$ defined and studied in \cite{LS}. Let us recall its definition and some properties.
Let $\{e_n\}$ be the canonical basis of the space $\ell_1$. Then
$$\Delta=\bigcup_{n=1}^\infty \co\{0,e_n,e_{n+1}\}.$$
The space $\Delta$ is considered with the metric inherited from $\ell_1$.
One of its important properties is contained in the following lemma.

\begin{lemma}\label{lm:Delta}  {\rm(\cite[Section 5]{LS})} There exists a Lipschitz mapping $g:\Delta\to\Delta$ such that $\inf_{x\in\Delta}\|x-g(x)\|>0$, hence $g$ has no approximate fixed point net.
\end{lemma}

Let us start the proof of Theorem~\ref{thm:M2-1}. Let $(X,\T)$, $C$ and $\rho_0$ satisfy the assumptions. Without loss of generality we can suppose that $0\in C$.
It follows that there exists $\delta>0$ such that for any finite-dimensional subspace $F\subset X$ there is some $x\in C$ with $\dist_{\rho_0}(x,F)>\delta$. So, we can construct by induction a sequence $(x_n)$ in $C$ such that

\begin{itemize}
	\item[(i)] $\rho_0(x_1)>\delta$,
	\item[(ii)] $\dist_{\rho_0}(x_{n+1},\spann\{x_1,\dots,x_n\})>\delta$ for any $n\in\N$.
\end{itemize}

Further, we will prove the following lemma.

\begin{lemma} Let $\rho$ be a continuous seminorm on $X$ such that $\rho\ge\rho_0$. Then there are constants $m>0$ and $M>0$ such that
\begin{equation}
\label{eq:e1} m\sum_{i=1}^4 |\alpha_i|\le \rho\left(\sum_{i=1}^4\alpha_i x_{k_i}\right)\le M\sum_{i=1}^4 |\alpha_i|
\end{equation}
whenever $k_1<k_2<k_3<k_4$ are natural numbers and $\alpha_i$ are scalars for $i\in\{1,\dots,4\}$.\end{lemma}

\begin{proof}
Since $C$ is bounded, there is $M>0$ such that $\rho(x_n)\le M$ for each $n\in\N$. Without loss of generality we can suppose that $M>\delta$. Let $k_1<k_2<k_3<k_4$ be natural numbers and $\alpha_i$ be scalars for $i\in\{1,\dots,4\}$. Then clearly
$$\rho\left(\sum_{i=1}^4\alpha_i x_{k_i}\right)\le M\sum_{i=1}^4 |\alpha_i|,$$
which proves the second inequality in \eqref{eq:e1}.
To show the first one we set $$c_i=\frac1{2^{2i+1}}\cdot\left(\frac\delta M\right)^{i-1}\mbox{ for }i=1,\dots,4.$$  Then $\sum_{i=1}^4 c_i<1$ and hence there is some $i\in\{1,\dots,4\}$ such that
$|\alpha_i|\ge c_i\sum_{i=1}^4 |\alpha_i|$. Let $i_0$ be the greatest such $i$. Then
$$\begin{aligned}\rho\left(\sum_{i=1}^4\alpha_i x_{k_i}\right)&\ge
\rho\left(\sum_{i=1}^{i_0}\alpha_i x_{k_i}\right)-\rho\left(\sum_{i=i_0+1}^4\alpha_i x_{k_i}\right)
\ge \delta|\alpha_{i_0}|-M\sum_{i=i_0+1}^4|\alpha_i|
\\&\ge\left(\delta c_{i_0}-M\sum_{i=i_0+1}^4 c_i\right)\sum_{i=1}^4|\alpha_i|
=\delta\left(c_{i_0}-\frac M\delta\sum_{i=i_0+1}^4 c_i\right)\sum_{i=1}^4|\alpha_i|
\\&\ge \frac12 c_{i_0}\delta\sum_{i=1}^4|\alpha_i|
\ge\frac1{32}\cdot\frac{\delta^4}{M^3}\sum_{i=1}^4|\alpha_i|.\end{aligned}$$
The first inequality follows from the triangle inequality, the second one follows from the condition (ii) above using the fact that $\rho\ge\rho_0$ and from the choice of $M$. The third one follows from the choice of $i_0$, the next equality is obvious.
Last two inequalities follow from the choice of  the constants $c_1,\dots,c_4$. So, we can set
$$m=\frac1{32}\cdot\frac{\delta^4}{M^3}$$
and the proof is completed.
\end{proof}

We continue proving the theorem. Set
$$D=\bigcup_{n\in\N}\co\{0,x_n,x_{n+1}\}.$$
Then $D\subset C$. {Moreover}, by the previous lemma we get immediately:

\begin{corollary} \
\begin{itemize}
	\item $(D,\rho_0)$ is bi-Lipschitz isomorphic to $\Delta$.
	\item For any continuous seminorm $\rho$ on $X$ satisfying $\rho\ge\rho_0$ the identity on $D$ is $\rho_0$-to-$\rho$ bi-Lipschitz.
\end{itemize}
\end{corollary}

We have even the following:

\begin{lemma}\label{lm-lip} The identity on $D$ is $\rho_0$-to-$\rho$ Lipschitz for any continuous seminorm $\rho$ on $X$. In particular, it is $\T$-to-$\rho_0$ uniform homeomorphism.
\end{lemma}

\begin{proof}  Let $\rho$ be a continuous seminorm. Set $\rho_1=\rho+\rho_0$. Then $\rho_1$ is a continuous seminorm satisfying $\rho_1\ge\rho_0$. So, by the above corollary, the identity on $D$ is $\rho_0$-to-$\rho_1$ Lipschitz. Hence, it is a fortiori $\rho_0$-to-$\rho$ Lipschitz.

It follows that the identity on $D$ is $\rho_0$-to-$\T$ uniformly continuous. Conversely, since $\rho_0$ is a continuous seminorm, the identity is clearly $\tau$-to-$\rho_0$ uniformly continuous.
\end{proof}

The next step is the following lemma:

\begin{lemma} There is a uniformly continuous retraction $r:X\to D$ which is also $\rho_0$-to-$\rho_0$ Lipschitz.
\end{lemma}

\begin{proof}
Let $\jmath:(X,\T)\to (X,\rho_0)$ be the identity mapping. Further,
let $Y$ be the quotient $(X,\rho_0)/\rho_0^{-1}(0)$. Then $Y$ is a normed space. Let $q:(X,\rho_0)\to Y$ denote the quotient mapping.
Then $q\circ \jmath$ is a continuous linear mapping, hence it is uniformly continuous.

Moreover, $\jmath$ is a uniform homeomorphism of $D$ onto $\jmath(D)$ and
$q$ is an isometry of $\jmath(D)$ onto $q(\jmath(D))$. It follows that $q(\jmath(D))$ is bi-Lipschitz isomorphic to $\Delta$. So, by \cite[Proposition 1]{LS} there is a Lipschitz retraction
$r_0: Y\to q(\jmath(D))$. Then
$$r= (q\circ \jmath|_D)^{-1}\circ r_0\circ q\circ \jmath$$
is a uniformly continuous retraction of $(X,\tau)$ onto $D$.
Moreover, $r$ is $\rho_0$-to-$\rho_0$ Lipschitz as $q\circ\jmath$ has this property, $r_0$ is Lipschitz and $q\circ\jmath|_D$ is an isometry of $(D,\rho_0)$ into $Y$.
\end{proof}

Now we are ready to finish the proof:

Since $(D,\rho_0)$ is bi-Lipschitz isomorphic to $\Delta$, by Lemma~\ref{lm:Delta} there is a Lipschitz map $g_0:(D,\rho_0)\to (D,\rho_0)$ without approximate fixed point net with respect to $\rho_0$. Further, let $r$ be the retraction from the previous lemma. Set
$$f=g_0\circ r|_C.$$
Then $f$ is a selfmap of $C$ which is $\rho_0$-to-$\rho_0$ Lipschitz. By Lemma~\ref{lm-lip} it is $\rho_0$-to-$\rho$ Lipschitz for any continuous seminorm $\rho$ (as $f(C)\subset D$). Further it has no approximate fixed point net. To see this, it is enough to find an $\varepsilon>0$ such that $\rho_0(x-f(x))\ge\varepsilon$ for each $x\in C$. This can be done imitating the proof of \cite[Lemma on p. 634]{LS}:

Let $L$ be the $\rho_0$-to-$\rho_0$ Lipschitz constant of $f$.
Let $\eta=\inf_{x\in q(\jmath(C))}\|x-g_0(x)\|$. Set
$$\varepsilon=\frac{\eta}{L+2}.$$
Fix an arbitrary $x\in C$. If $\dist_{\rho_0}(x,D)\ge\varepsilon$,
then $\rho_{0}(x-f(x))\ge\varepsilon$ as $f(x)\in D$.
If $\dist_{\rho_0}(x,D)<\varepsilon$, find $y\in D$ with $\rho_0(x-y)<\varepsilon$. Then
$$\begin{aligned}\rho_0(x-f(x))&\ge \rho_0(y-f(y))-\rho_0(x-y)-\rho_0(f(x)-f(y))
\ge \eta-(1+L)\rho_0(x-y) \\&> (L+2)\varepsilon-(1+L)\varepsilon=\varepsilon.\end{aligned}$$
This completes the proof.\hfill $\square$

\bigskip

\noindent{\bf Acknowledgment.} The research described in this paper was started while the first author was visiting Universidade Federal do Amazonas, Manaus-AM, during the period: August 2011 to November 2011. He wishes to thank Professor Renato Tribuzy and all the other members at the Department of Mathematics for their hospitality. We wish also to thank S.~Park for kindly sending us a copy of his paper \cite{Park}.

\vskip1truecm

\scshape

\noindent Cleon S. Barroso and Michel P. Rebou\c cas

\noindent Universidade Federal do Cear\'a

\noindent {Departamento} de Matem\'atica

\noindent Av. Humberto Monte S/N, 60455-760, Bl 914

\noindent E-mail address: \qquad {\tt cleonbar@mat.ufc.br}\quad and \quad {\tt michelufc@yahoo.com.br}

\scshape
\bigskip

\noindent Ond\v{r}ej F. K. Kalenda

\noindent Charles University

\noindent Faculty of Mathematics and Physics, Department of Mathematical Analysis

\noindent Sokolovsk\'a 83, 186 75, Praha 8, Czech Republic

\noindent E-mail address: \qquad {\tt kalenda@karlin.mff.cuni.cz}


\begin{thebibliography}{00}

\bibitem{AGP} O. Arino, S. Gautier and J. P. Penot, A fixed point theorem for sequentially continuous mappings
with application to ordinary differential equations, Funkcial. Ekvac., {\bf 27} (1984), 273–-279.

\bibitem{Barroso} C. S. Barroso, The approximate fixed point property in Hausdorff topological vector spaces and applications, Discrete Cont. Dyn. Syst., {\bf 25} (2009), 467--479.

\bibitem{Barroso-Lin} C. S. Barroso and P-K. Lin, On the weak approximate fixed point property, J. Math. Anal. Appl., {\bf 365} (2010), 171--175.

\bibitem{BKL} C. S. Barroso, O. F. K. Kalenda and P.-K. Lin, On the approximate fixed point property in abstract spaces, To appear in Math. Z. (2011) DOI: 10.1007/s00209-011-0915-6.


\bibitem{Fan61} K. Fan, A Generalization of Tychonoff's Fixed Point
Theorem, Math. Annal. {\bf 142} (1961), 305--310.

\bibitem{GJZ} L. Gajek, J. Jachymski, D. Zagrodny, Fixed point and approximate fixed point theorems for non-affine maps. J. Appl. Anal. {\bf 1} (1995), no. 2, {205-–211}.

\bibitem{Hadzic1} O. Had\v{z}i\'c, Some fixed point and almost fixed point theorems for multivalued mappings in topological vector spaces, Nonlinear Anal. {\bf 5} (1981), 1009--1019.

\bibitem{Hadzic2} O. Had\v{z}i\'c, Almost fixed point and best approximations theorems in $H$-spaces, Bull. Austral. Math. Soc. {\bf 53} (1996), no. 3, 447--454.

\bibitem{HV} M. Hazewinkel and M. van de Vel, On almost-fixed-point theory, Canad. J. Math. {\bf 30} (1978), 673--699,

\bibitem{Himmelberg} C. J. Himmelberg, Fixed points of compact multifunctions, J. Math. Anal. Appl. {\bf 38} (1972),
205–-207.

\bibitem{Idzik1} A. Idzik, On $\gamma$-almost fixed point theorems. The single-valued case, Bull. Polish Acad. Sci. Math. {\bf 35} (1987), no. 7--8, 461--464.

\bibitem{Idzik2} A. Idzik, Almost fixed point theorems, Proc. Amer. Math. Soc. {\bf 104} (1988), {779-–784}.

\bibitem{klee} V. Klee, Some topological properties of convex sets, Trans. Amer. Math. Soc. {\bf 78} (1955), 30-45.

\bibitem{KKM} B. Knaster, C. Kuratowski and S. Mazurkiewicz, Ein Beweis des Fixpunkt-satzes f\"ur $n$-dimensionale simplexe, Fundamenta Math. {\bf 14}, (1929), 132--137.

\bibitem{LS} P.-K. Lin and Y. Sternfeld, Convex sets with the Lipschitz fixed point property are compact, Proc. Amer. Math. Soc., {\bf 93} (1985), 633--639.


\bibitem{Park} S. Park, D. H. Tan, Remarks on the Schauder-Tychonoff fixed point theorem, Vietnam J. Math. {\bf 28} (2000), no. 2, 127-–132.

\bibitem{Park2} S. Park, Almost fixed points of multimaps having totally bounded ranges, Nonlinear Anal. {\bf 51} (2002), no. 1, 1--9.



\end{thebibliography}
\end{document}